\documentclass{amsart}
\usepackage{ amsmath, amsthm, amsfonts, hyperref, graphicx, ifpdf}
\usepackage[dvipsnames,usenames]{color}
\hypersetup{
   bookmarks=true,         
   unicode=false,          
   pdftoolbar=true,        
   pdfmenubar=true,        
   pdffitwindow=false,     
   pdfstartview={},    
   pdftitle={},    
   pdfauthor={Longzhi Lin},     
   pdfsubject={},   
   pdfcreator={Longzhi Lin},   
   pdfproducer={Longzhi Lin}, 
   pdfkeywords={}, 
   pdfnewwindow=true,      
   colorlinks=true,       
   linkcolor=blue,          
   citecolor=blue,        
   filecolor=magenta,      
   urlcolor=MidnightBlue          
}

\newcommand{\N}{{\mathbb N}}

\newtheorem{theorem}{Theorem}[section]
\newtheorem{lemma}[theorem]{Lemma}
\newtheorem{proposition}[theorem]{Proposition}
\newtheorem{corollary}[theorem]{Corollary}

\theoremstyle{definition}
\newtheorem{definition}[theorem]{Definition}

\theoremstyle{remark}
\newtheorem{remark}[theorem]{Remark}

\numberwithin{equation}{section}

\begin{document}
\title[Estimates for the energy density]{Estimates for the energy density of critical points of a class of conformally invariant variational problems}
\author{Tobias Lamm}
\address[T.~Lamm]{Institut f$\ddot{\text{u}}$r Mathematik\\Goethe-Universit$\ddot{\text{a}}$t Frankfurt\\ Robert-Mayer-Str. 10\\D-60054 Frankfurt, Germany}
\email{lamm@math.uni-frankfurt.de}
\author{Longzhi Lin}
\address[L.~Lin]{Department of Mathematics\\Rutgers University\\110 Frelinghuysen Road\\Piscataway, NJ 08854-8019\\USA}
\email{lzlin@math.rutgers.edu}
\date{\today}

\begin{abstract}
We show that the energy density of critical points of a class of conformally invariant variational problems with small energy on the unit $2$-disk $B_1\subset\mathbb{R}^2$ lies in the local Hardy space $h^1(B_1)$. As a corollary we obtain a new proof of the energy convexity and uniqueness result for weakly harmonic maps with small energy on $B_1$.
\end{abstract}

\maketitle

\section{Introduction}
Let $\mathcal{N}$ be a closed (i.e. compact and without boundary) $C^2$ Riemannian submanifold of $\mathbb{R}^n$. The \textit{Dirichlet} energy of a $W^{1,2}$ map $u: B_1\to \mathcal{N}$ from the $2$-dimensional unit disk $B_1=B_1(0)\subset \mathbb{R}^2$ is defined by
\begin{equation}
\label{energy} E(u)\,=\,\frac{1}{2}\,\int_{B_1}\,|\nabla u|^2 dx\,.
\end{equation}
A weakly harmonic map $u$ from $B_1$ into $\mathcal{N}\hookrightarrow \mathbb{R}^n$ is a map in $W^{1,2}(B_1, \mathbb{R}^n)$, which takes values almost everywhere in $\mathcal{N}$ and solves the Euler-Lagrange equation
\begin{equation}\label{EL}
(\Delta u)^{\top} \,=\, 0\,,
\end{equation}
where $u=(u^1,...,u^n), \Delta= \sum_{i=1}^2\frac{\partial^2}{\partial x_i^2}$ is the Laplacian in $\mathbb{R}^2$ and the superscript $\top$ denotes the tangential part of a vector. The Euler-Lagrange equation \eqref{EL} can equivalently be written as
\begin{equation}\label{HM}
-\Delta u\,=\, A(u)(\nabla u, \nabla u)\,,
\end{equation}
where $A(u)$ is the second fundamental form of $\mathcal{N}\hookrightarrow\mathbb{R}^n$ at the point $u$. We point out that it is not important that the source domain $B_1\subset \mathbb{R}^2$ is flat in this work. Indeed, using the conformal invariance of \eqref{energy} resp. \eqref{HM}, we could instead work on any domain that is conformally equivalent to $B_1$, for example, a simply-connected open subset of any general Riemannian surface.

Using a so called Coulomb or moving frame, H\'elein (see e.g. \cite{He1}) proved the interior regularity of weakly harmonic maps. Later, Qing \cite{Q} showed continuity up to the boundary in the case of continuous boundary data based on H\'elein's technique. More recently, Rivi\`ere \cite{Riv1} succeeded in writing the Euler-Lagrange equation of every conformally invariant Lagrangian (which includes \eqref{HM}) in the form:
\begin{equation}\label{CIPDE1}
-\Delta u^i\,=\, \Omega^i_{j}\cdot \nabla u^j \quad i=1,2,...,n \quad\text{or}\quad -\Delta u\,=\, \Omega\cdot \nabla u\,,
\end{equation}
with $\Omega=(\Omega^i_j)_{1\leq i,j\leq n}\in L^2(B_1, so(n)\otimes\wedge^1\mathbb{R}^2)$, i.e. $\Omega^i_{j}=-\Omega^j_{i}$. Here and throughout the paper, the Einstein summation convention is used. In particular, this special form of the nonlinearity enabled Rivi\`ere to obtain a conservation law for this system of PDE's (see \eqref{AandB2} below), which is accomplished via a technique that we call Rivi\`ere's gauge decomposition, see Section \ref{RiviereGauge}. Rivi\`ere's gauge decomposition will also be the main tool of our work.

Equation \eqref{CIPDE1} generalizes a number of interesting equations, including the harmonic map equation \eqref{HM}), the $H$-surface equation and, more generally, the Euler-Lagrange equation of any conformally invariant elliptic Lagrangian which is quadratic in the gradient. We remark that the harmonic map equation \eqref{HM} can be written in the form of \eqref{CIPDE1} if we set
\begin{equation}\label{SFF}
\Omega^i_j:=\,[A^i(u)_{j,l}-A^j(u)_{i,l}]\nabla u^l\,.
\end{equation}
Using the conservation law mentioned above, Rivi\`ere proved the (interior) continuity of any $W^{1,2}$ weak solution $u$ of \eqref{CIPDE1}.

The underlying reason for Rivi\`ere's argument to work is that equation \eqref{CIPDE1} can be rewritten in the form
\begin{equation}\label{TEMPLL3}
\text{div}(A\nabla u)\,=\,\nabla^\perp B \cdot\nabla u\,,
\end{equation}
where $A \in L^{\infty}\,\cap\, W^{1,2}(B_1, Gl_n(\mathbb{R}))$ and $B\in W_0^{1,2}(B_1, M_n(\mathbb{R}))$. Here and in what follows we let $\nabla = (\partial_{x}, \partial_{y})$ be the gradient and $\nabla^{\perp} = (-\partial_{y}, \partial_{x})$ denote the orthogonal gradient (i.e., $\nabla^{\perp}$ is the $\nabla$-operator rotated by $\pi/2$). The right hand side of this new equation \eqref{TEMPLL3} lies in the Hardy space $\mathcal{H}^1$ by a result of Coifman, Lions, Meyer and Semmes \cite{CLMS}. Moreover, using a Hodge decomposition argument, one can show that $u$ lies locally in $W^{2,1}$ which embeds into $C^0$ in two dimensions. For details see Section \ref{RiviereGauge}. The key to this fact is a special ``compensation phenomenon'' for Jacobian determinants (e.g. the right hand side of \eqref{TEMPLL3}), which was first observed by Wente \cite{W}, see also Lemma \ref{Wen}. These Wente type estimates have many interesting applications, see e.g. \cite{W,BC,Ta2,CLMS,He1,Riv1,Riv2}. Recently, Sharp and Topping \cite{ST}, proved some interesting interior estimates for solutions of \eqref{CIPDE1} and they even obtained a compactness result for sequences of solutions.

A similar approach was used by Rivi\`ere and the first author \cite{LamR} in order to show the regularity of solutions of certain fourth order systems of PDE's which include intrinsic and extrinsic biharmonic maps.

Another interesting question is the uniqueness of weakly harmonic maps. Colding and Minicozzi recently showed the following energy convexity result for weak solutions of \eqref{HM} with small energy on $B_1$.
\begin{theorem}[\cite{CM1}]\label{THM1}
There exists constant $\varepsilon_0>0$ depending only on $\mathcal{N}$ such that if $u, v \in W^{1,2}(B_1, \mathcal{N})$ with $u|_{\partial B_1}=v|_{\partial B_1}, E(u)\leq\varepsilon_0$, and $u$ is weakly harmonic, then we have the energy convexity
\begin{equation}\label{Convexity1}
\frac{1}{2}\int_{B_1} |\nabla v-\nabla u|^2\,\leq\, \int_{B_1} |\nabla v|^2-\int_{B_1}|\nabla u|^2\,.
\end{equation}
\end{theorem}
An immediate corollary of Theorem \ref{THM1} is the uniqueness of solutions of the Dirichlet problem for weakly harmonic maps with small energy on $B_1$.
\begin{corollary}[\cite{CM1}] \label{COR1}
There exists $\varepsilon_0>0$ such that for all weakly harmonic maps $u, v \in W^{1,2}(B_1, \mathcal{N})$ with energy $E(u),E(v)\leq\varepsilon_0$ and $u=v$ on $\partial B_1$ we have $u = v$ in $B_1$.
\end{corollary}

The key ingredient in Colding and Minicozzi's proof of Theorem \ref{THM1} is a special Jacobian structure of the squared norm of a certain holomorphic function that H\'elein constructed in \cite{He1}. This holomorphic function bounds the energy density $|\nabla u|^2$ from above. In particular, although Colding and Minicozzi didn't use this fact explicitly, this implies that the energy density $|\nabla u|^2$ lies in the local Hardy space ${h}^1(B_1)$. This implies that for weakly harmonic maps with small energy the quantity $|\Delta u|$ ($\leq C|\nabla u|^2$ by the harmonic map equation \eqref{HM}) lies in the local Hardy space.
This improved information is the key fact in order to prove \eqref{Convexity1}.

We remark that global estimates on the whole disk are needed in order to prove results such as the energy convexity. On the other hand, without imposing any regularity on the boundary data, global estimates are very difficult to obtain. In this paper we show that a global estimate in the local Hardy space $h^1(B_1)$ (see Definition \ref{10003}, cf. \cite{Sem}) for the energy density $|\nabla u|^2$ is valid for a more general class of non-linear systems of second order elliptic PDE's. Throughout the paper $\varepsilon_0>0$ is assumed to be sufficiently small and $C$ denotes a universal constant that depends only on $\mathcal{N}$ unless otherwise stated. Our main result is the following

\begin{theorem}\label{THMTL}
Let $\mathcal{N} \hookrightarrow \mathbb{R}^n$ be a closed $C^2$ Riemannian submanifold with a $C^{2,1}$ metric. There exists a constant $\varepsilon_0>0$, depending only on $\mathcal{N}$, such that if $u\in W^{1,2}(B_1, \mathcal{N})$ with $E(u)\leq \varepsilon_0$ is a weak solution of
\begin{equation}\label{RiviereE}-\Delta u = \Omega \cdot \nabla u\end{equation}
where $\Omega=(\Omega^i_j)_{1\leq i,j\leq n}\in L^2(B_1, so(n)\otimes\wedge^1\mathbb{R}^2)$ and if
\begin{align}
&\Omega^i_j = \sum_{l=1}^n f^i_{jl} \nabla u^l + g^i_{jl} \nabla^{\perp} u^l \text{ with } f^i_{jl} = - f^j_{il}, g^i_{jl} = - g^j_{il}\,;\\
&\|f\|_{L^{\infty}(B_1)}+ \|\nabla f\|_{L^2(B_1)}+ \|g\|_{L^{\infty}(B_1)}+\|\nabla g\|_{L^2(B_1)}\leq C\,,
\end{align}
then we have
\begin{equation}\label{MainTHMEst}
|\nabla u|^2 \in h^1(B_1) \quad\quad \text{and}\quad  \quad\||\nabla u|^2 \|_{h^1(B_1)} \leq C \int_{B_1}|\nabla u|^2\leq C\varepsilon_0\,.
\end{equation}
\end{theorem}

We point out that Theorem \ref{THMTL} applies to critical points of a large class of elliptic conformally invariant variational problems in dimension two (see \cite[Theorem I.2]{Riv1}). More precisely we have

\begin{corollary}\label{THMTL2} Let $\mathcal{N} \hookrightarrow \mathbb{R}^n$ be a closed $C^2$ Riemannian submanifold with a $C^{2,1}$ metric. Let $\omega$ be a $C^{1,1}$ $2$-form on $\mathcal{N}$ such that the Lipschitz norm of $d\omega$ is bounded on $\mathcal{N}$. Then every critical point in $W^{1,2}(B_1, \mathcal{N})$ of the Lagrangian
\begin{equation}
F(u)\,=\, \int_{B_1} \left[|\nabla u|^2 + \omega(u)(\partial_x u, \partial_y u)\right]dx \wedge dy
\end{equation}
satisfies equation \eqref{RiviereE} with
\begin{equation}
\Omega^i_j\,=\, [A^i(u)_{j,l}-A^j(u)_{i,l}]\nabla u^l + \frac{1}{4}[\lambda^i(u)_{j,l}-\lambda^j(u)_{i,l}]\nabla^{\perp} u^l\,,
\end{equation}
where $A$ and $\lambda$ are in $C^1(\mathcal{N}, M_n(\mathbb{R})\otimes\wedge^1\mathbb{R}^2))$ satisfying
$$ \sum_{j=1}^n A^j_{i,l} \nabla u^j = 0 \quad \text{and}\quad \lambda^i_{j,l} = d(\pi^{\ast}_{\mathcal{N}} \omega)(\varepsilon_i,\varepsilon_j,\varepsilon_l)$$
where $\pi_{\mathcal{N}}$ is the orthogonal projection onto $\mathcal{N}$ in a small tubular neighborhood of $\mathcal{N}$ and $\{\varepsilon_i\}_{i=1,...,n}$ is the canonical basis of $\mathbb{R}^n$. If additionally we have $E(u)\leq \varepsilon_0$ then
\begin{equation}
|\nabla u|^2 \in h^1(B_1) \quad\quad \text{and}\quad \quad \||\nabla u|^2 \|_{h^1(B_1)} \leq  C \int_{B_1}|\nabla u|^2\leq C\varepsilon_0\,.
\end{equation}
\end{corollary}

\begin{remark}
We remark that Corollary \ref{THMTL2} includes the case of the prescribed mean curvature equation in $\mathbb{R}^3$:
\begin{equation}
-\Delta u\,=\, -2 H(u) \partial_x u \wedge \partial_y u
\end{equation}
if $\|H\|_{W^{1,\infty}(B_1)} < \infty$.
\end{remark}

\begin{remark}
In \cite{Lin12}, using similar techniques, the second author proved an energy convexity along the harmonic map heat flow with small initial energy and fixed boundary data on $B_1$ . In particular this yields that such a harmonic map heat flow converges uniformly in time strongly in the $W^{1,2}$-topology, as time goes to infinity, to the unique limiting harmonic map.
\end{remark}
Our approach to the proof of Theorem \ref{THMTL} is based on Rivi\`ere's gauge decomposition. An immediate application of Theorem \ref{THMTL} is a new proof of Colding and Minicozzi's Theorem \ref{THM1} and Corollary \ref{COR1}.

The paper is organized as follows. In Section \ref{Sec1} we present some heuristic arguments and elaborate on the motivation of this paper. In Section \ref{RiviereGauge} we review the main tool of our proof, namely, Rivi\`ere's gauge decomposition technique. In Sections \ref{MatrixB} and \ref{MatrixU} we show improved estimates for the matrices $B$ and $P$ which are the two main ingredients of our proof. We finish the proof of our main theorem in Section \ref{LastSec}. In Appendix \ref{APPDEN} we include some results about local Hardy spaces.

\section{Heuristic arguments and motivation}\label{Sec1}
We present some heuristic arguments and sketch the basic idea of the proof of Theorem \ref{THM1} in this section. In order to prove the energy convexity result \eqref{Convexity1}, i.e.,
\begin{equation}
\frac{1}{2}\int_{B_1} |\nabla v-\nabla u|^2\,\leq\, \int_{B_1} |\nabla v|^2-\int_{B_1}|\nabla u|^2\,,
\end{equation}
it suffices to show
\begin{equation}\label{TEMP1}
\Psi\,\geq\, -\frac{1}{2}\int_{B_1} |\nabla v-\nabla u|^2\,,
\end{equation}
where (using that $u|_{\partial B_1}=v|_{\partial B_1}$ and the harmonic map equation \eqref{HM})
\begin{align}
 \Psi := &\int_{B_1} |\nabla v|^2-\int_{B_1}|\nabla u|^2-\int_{B_1} |\nabla v-\nabla u|^2 \notag\\
 = &\, 2\int_{B_1} \langle \nabla v - \nabla u ,\nabla u\rangle   \notag\\
 = &\,2\int_{B_1} \langle v - u ,  \,A(u)(\nabla u, \nabla u)\rangle\,.\label{Convexity2000}
\end{align}
Now we note that for any $p, q \in \mathcal{N}$, there exists a constant $C>0$, depending only on $\mathcal{N}$, such that $\left|(p-q)^{\perp}\right|\leq C |p-q|^2$, where the superscript $\perp$ denotes the normal component of a vector (see e.g. \cite[Lemma A.1]{CM2}). Therefore, using $A(u)(\nabla u, \nabla u) \perp T_u \mathcal{N}$ and the Cauchy-Schwarz inequality, \eqref{Convexity2000} yields
\begin{equation}
\Psi\geq  - C\int_{B_1} |(v - u)^\perp||\nabla u|^2 \geq   - C\int_{B_1} |v - u|^2|\nabla u|^2\,.
\end{equation}

Since $\varepsilon_0$ can always be chosen to be sufficiently small, we know that \eqref{Convexity1} will be achieved if we can show the following lemma.

\begin{lemma}\label{GradEst}
Let $u,v$ be as in Theorem \ref{THM1}. Then we have
\begin{equation}\label{Convexity4000}
\int_{B_1}|v - u|^2|\nabla u|^2\leq C\int_{B_1}|\nabla u|^2\int_{B_1}|\nabla v-\nabla u|^2 \leq C \varepsilon_0\int_{B_1}|\nabla v-\nabla u|^2\,.
\end{equation}
\end{lemma}

In the following lemma we verify \eqref{Convexity4000} under some extra assumption.
\begin{lemma}\label{GradEst1}
Let $u,v$ be as in Theorem \ref{THM1} and suppose that we can find a solution $\psi \in W^{1,2}_0 \cap L^\infty(B_1)$ of
\begin{equation}\label{PSI1}
\left\{
   \begin{aligned}
     \Delta \psi\,& = \,|\nabla u|^2 && \text{in }\, B_1\,, \\
    \psi\,&=\, 0 && \text{on }\, \partial B_1\,,\\
   \end{aligned}
 \right.
\end{equation}
which satisfies
\begin{equation}\label{PSI2}
 \|\psi\|_{L^{\infty}(B_1)}+ \|\nabla \psi\|_{L^2(B_1)}\,\leq\, C\int_{B_1}|\nabla u|^2\leq C\varepsilon_0\,.
\end{equation}
Then Lemma \ref{GradEst} holds.
\end{lemma}
\begin{proof}
The following proof is taken from \cite{CM1}. Substituting \eqref{PSI1} into the left-hand side of \eqref{Convexity4000} yields (using also that $v = u$ on $\partial B_1$)
\begin{align}\label{PSI3}
\int_{B_1}|v - u|^2&|\nabla u|^2= \int_{B_1}
|v - u|^2\Delta\psi\leq\int_{B_1} |\nabla |v - u|^2||\nabla
\psi| \notag\\
\leq\,&2\left(\int_{B_1}|\nabla v-\nabla u|^2\right)^{1/2}\left(\int_{B_1}\,|v - u|^2|\nabla
\psi|^2\right)^{1/2},
\end{align}
where we have applied Stokes' theorem to $\text{div}(|v-u|^2 \nabla\psi)$ and
used the Cauchy-Schwarz inequality. Now applying Stokes' theorem to $\text{div}(|v-u|^2\psi \nabla\psi)$ and using that $\Delta\psi\,\geq\,0$ and \eqref{PSI3}, we have
\begin{align}\label{PSI4}
\int_{B_1}|v-&u|^2|\nabla \psi|^2 \leq\,\int_{B_1}
|\psi|(|v-u|^2\Delta\psi+|\nabla |v-u|^2||\nabla
\psi|)\notag\\
\leq &\,4\|\psi\|_{L^{\infty}}\left(\int_{B_1}|\nabla v-\nabla u|^2\right)^{1/2}\left(\int_{B_1}|v-u|^2|\nabla \psi|^2\right)^{1/2},
\end{align}
and therefore
\begin{equation}\label{PSI5}
\left(\int_{B_1} |v-u|^2|\nabla \psi|^2\right)^{1/2}\leq
4\|\psi\|_{L^{\infty}}\left(\int_{B_1}|\nabla v-\nabla u|^2\right)^{1/2}.
\end{equation}
Finally, substituting \eqref{PSI5} back into \eqref{PSI3} and combining with \eqref{PSI2} (and choosing $\varepsilon_0$ sufficiently small of course) yields
\begin{equation}
\int_{B_1}|v - u|^2|\nabla u|^2\leq C \|\psi\|_{L^{\infty}}\int_{B_1}|\nabla v-\nabla u|^2\leq C\int_{B_1}|\nabla u|^2\int_{B_1}|\nabla v-\nabla u|^2\,,
\end{equation}
which is just \eqref{Convexity4000}.
\end{proof}

Therefore, everything boils down to validating the assumptions in Lemma \ref{GradEst1}, i.e., the existence of a function $\psi$ satisfying \eqref{PSI1} and \eqref{PSI2}.
In the light of the regularity Theorem \ref{CDSTHM} it will be sufficient to show that $|\nabla u|^2$ is in the local Hardy space $h^1(B_1)$ (with estimates) and this is exactly the claim of our main theorem \ref{THMTL}.

\section{Rivi\`ere's gauge decomposition}\label{RiviereGauge}
Following the strategy of Uhlenbeck in \cite{Uh}, Rivi\`ere \cite{Riv1} used the algebraic feature of $\Omega$, namely $\Omega$ being antisymmetric, to construct $\xi\in W^{1,2}_{0}(B_1, so(n))$ and a gauge transformation matrix $P\in W^{1,2}\cap L^\infty(B_1, SO(n))$ (which pointwise almost everywhere is an orthogonal matrix in $\mathbb{R}^{n\times n}$) satisfying some good properties.
\begin{theorem}\label{PCloseTOID}(\cite[Lemma A.3]{Riv1}) There exist $\varepsilon>0$ and $C>0$ such that for every $\Omega$ in $L^2(B_1, so(n)\otimes\wedge^1\mathbb{R}^2)$ satisfying
$$
\int_{B_1}|\Omega|^2\,\leq\, \varepsilon\,,
$$
there exist $\xi\in W_0^{1,2}(B_1, so(n))$ and $P\in W^{1,2}(B_1, SO(n))$ such that
\begin{equation}\label{P-1}
\nabla^{\perp}\xi \,=\, P^{T}\nabla P + P^T\Omega P \text{ in} \,\,B_1 \quad\text{with}\quad \xi \,=\,0\text{ on} \,\,\partial B_1,
\end{equation}
and
\begin{equation}\label{P-2}
\|\nabla \xi\|_{L^2(B_1)}+ \|\nabla P\|_{L^2(B_1)} \,\leq\, C\|\Omega\|_{L^2(B_1)}\,.
\end{equation}
Here the superscript $T$ denotes the transpose of a matrix.
\end{theorem}
\begin{remark} Multiplying both sides of equation \eqref{P-1} by $P$ from the left gives that (with indices and $1\leq m, z\leq n$)
\begin{equation} \label{P-4}
\nabla P^i_j\,=\,P^i_m\nabla^{\perp} \xi^m_j - \Omega^i_z \,P^z_j,\quad  1\leq i,j \leq n\,.
\end{equation}
\end{remark}
\begin{remark}\label{PMATRIXX}
Besides Uhlenbeck's method there is another way to construct the gauge tranformation matrix $P$, namely one can minimize the energy functional
\begin{equation}
E(R) \,=\, \int_{B_1} \left|R^T\nabla R + R^T\Omega R\right|^2
\end{equation}
among all $R\in W^{1,2}(B_1,SO(n))$, see e.g. \cite{Cho} and \cite{Sc}.
\end{remark}

Another key result from Rivi\`ere's work is the following theorem.
\begin{theorem}(\cite[Theorem I.4]{Riv1})\label{PCloseTOID299} There exist $\varepsilon>0$ and $C>0$ such that for every $\Omega$ in $L^2(B_1, so(n)\otimes\wedge^1\mathbb{R}^2)$ satisfying
$$
\int_{B_1}|\Omega|^2\,\leq\, \varepsilon\,,
$$
there exist $\widehat{A} \in W^{1,2}\,\cap\, C^{0}(B_1, Gl_n(\mathbb{R}))$, $A = (\widehat{A}+ Id)\,P^{T} \in L^{\infty}\,\cap\, W^{1,2}(B_1, Gl_n(\mathbb{R}))$ and $B\in W_0^{1,2}(B_1, M_n(\mathbb{R}))$ such that
\begin{equation}\label{A-B-1}
\nabla A - A \Omega\,=\, \nabla^{\perp} B
\end{equation}
and
\begin{equation}\label{A-B-2}
\|\widehat{A}\|_{W^{1,2}(B_1)} + \|\widehat{A}\|_{L^\infty(B_1)} + \|B\|_{W^{1,2}(B_1)}\,\leq\, C \|\Omega\|_{L^2(B_1)}\,.
\end{equation}
\end{theorem}
Combining \eqref{A-B-1} with \eqref{CIPDE1} gives the conservation law
 \begin{equation}\label{AandB2}
    \text{div }(A\nabla u+B\nabla^{\perp}u)=0.
   \end{equation}
In particular, \eqref{AandB2} yields
\begin{equation}\label{Regularity22}
  \left\{
   \begin{aligned}
    \text{div }( A\nabla u )  &= -\nabla B \cdot\nabla^\perp u \\
    \text{curl }( A\nabla u ) &= \nabla^\perp A \cdot\nabla u\,.
   \end{aligned}
 \right.
   \end{equation}
Now using the Hodge decomposition (see e.g. \cite[Corollary 10.5.1]{IM}) we get the existence of $D \in W_0^{1,2}(B_1, \mathbb{R}^n)$ and $E \in W^{1,2}(B_1, \mathbb{R}^n)$ such that
\begin{equation}\label{DandE1}
A\nabla u = \nabla D + \nabla^{\perp}E
\end{equation}
with
\begin{align*}
||\nabla D||_{L^2(B_1)}+||\nabla E||_{L^2(B_1)} \le C ||\nabla u||_{L^2(B_1)}.
\end{align*}
By \eqref{Regularity22} we have
\begin{equation}\label{DandE2}
  \left\{
   \begin{aligned}
    \Delta D  &= -\nabla B \cdot\nabla^\perp u \\
    \Delta E  &= \nabla^\perp A \cdot\nabla u\,.
   \end{aligned}
 \right.
   \end{equation}
Then by the results of \cite{CLMS} and via an extension argument, Rivi\`ere obtained that
\begin{equation}\label{DandE4}
D, E \in W^{2,1}_{loc}(B_1)\quad\text{and therefore}\quad u \in W^{2,1}_{loc}(B_1) \hookrightarrow C^0(B_1).
\end{equation}
Combining the fact that $D = 0$ on $\partial B_1$ with Wente's lemma \ref{Wen} and Remark \ref{Wen2} moreover yields
\begin{equation}\label{EstimateofD}
\|D\|_{L^\infty(B_1)}  + \|\nabla D\|_{L^{2,1}(B_1)}  \leq C\int_{B_1}|\nabla u|^2\leq C\varepsilon_0,
\end{equation}
where the Lorentz space $L^{2,1}(B_1)$ is defined as the set of all measurable functions such that the norm
\[
||u||_{L^{2,1}(B_1)}:=\int_0^\infty |\{x\in B_1:\,\ |u(x)|>t\}|^{1/2}  dt
\]
is finite.

\section{Hidden Jacobian structures}\label{4}
Our main observation in this section are two hidden Jacobian structures for $\Delta B$ and $\Delta P$ in the case that $\Omega$ is as in Theorem \ref{THMTL}.

\subsection{Improved global estimate on the matrix $B$}\label{MatrixB}

We first show that $B$ is close to the zero matrix if $E(u)\leq \varepsilon_0$ is sufficiently small. Note that in the case of harmonic maps into a round sphere, a straightforward calculation shows that
\[
 \nabla^\perp B^{i}_j = u^i \nabla u^j -u^j \nabla u^i
\]
and therefore
\[
 \Delta B^i_j =\nabla^\perp u^i \nabla u^j- \nabla^\perp u^j \nabla u^i.
\]
Combining this with the fact that $B=0$ on $\partial B_1$ and Wente's lemma \ref{Wen} yields
\[
 ||B||_{L^\infty(B_1)} \le C \int_{B_1}|\nabla u|^2 \le C \varepsilon_0.
\]
In the following we show that a similar result remains true for the class of systems satisfying the assumptions of Theorem \ref{THMTL}.

\begin{proposition}\label{XI}
Under the assumptions of Theorem \ref{THMTL} we have
\begin{equation}\label{EstimateofB}
\|B\|_{L^\infty(B_1)}\,\leq\, C\|\nabla u\|_{L^2(B_1)}\,\leq\,C \sqrt{\varepsilon_0}\,.\end{equation}
\end{proposition}
\begin{proof}
We recall that $\Omega$ is given by
$$\Omega^i_j = f^i_{jl} \nabla u^l + g^i_{jl} \nabla^{\perp} u^l\,.$$
Moreover $\|\Omega\|^2_{L^2(B_1)} \leq C\varepsilon_0$ by the assumptions.
In the following we let $\varepsilon_0$ be so small that Theorems \ref{PCloseTOID} and \ref{PCloseTOID299} apply. By \eqref{DandE1} we have
\begin{equation}\label{DandE3}
\nabla^\perp u = A^{-1}\nabla^\perp D - A^{-1}\nabla E\,.
\end{equation}
Namely, with indices we have
\begin{equation}\label{P3}
\nabla^\perp u^{l}\,=\, (A^{-1})^{l}_{m}\nabla^\perp D^m  - (A^{-1})^{l}_{m} \nabla E^m\,,\quad l=1,2,...,n .
\end{equation}
Taking the curl on both sides of equation \eqref{A-B-1} and using the above expression for $\nabla^\perp u$ yields (for $1\leq i, j\leq n$)
\begin{align}\label{DandE6}
\Delta B^i_j\,= &\,- \text{curl }(A^i_k \Omega^k_j)\,=\, - \text{curl }(A^i_k ( f^k_{jl} \nabla u^l + g^k_{jl} \nabla^{\perp} u^l))\notag\\
=& \,-\nabla^{\perp} (A^i_k f^k_{jl}) \cdot \nabla u^l  - \text{curl }(A^i_k g^k_{jl} ((A^{-1})^{l}_{m}\nabla^{\perp} D^m - (A^{-1})^{l}_{m} \nabla E^m))\notag\\
=&\,-\nabla^{\perp} (A^i_k f^k_{jl}) \cdot \nabla u^l + \nabla^{\perp} (A^i_k g^k_{jl}(A^{-1})^{l}_{m})\cdot \nabla E^m \\
&\,- \text{div }(A^i_k g^k_{jl} ((A^{-1})^{l}_{m}\nabla D^m).\notag
\end{align}
Next we let $F^i_j,G^i_j\in W^{1,2}_0(B_1)$ be solutions of
\begin{align*}
 \Delta F^i_j=& -\nabla^{\perp} (A^i_k f^k_{jl}) \cdot \nabla u^l + \nabla^{\perp} (A^i_k g^k_{jl}(A^{-1})^{l}_{m})\cdot \nabla E^m\ \ \text{resp.} \\
\Delta G^i_j=& - \text{div }(A^i_k g^k_{jl} ((A^{-1})^{l}_{m}\nabla D^m).
\end{align*}
Using Lemma \ref{Wen} we get the estimate
\[
 ||F^i_j||_{L^{\infty}(B_1)} \le C||\nabla u||_{L^2(B_1)}.
\]
Next we note that $||A^i_k g^k_{jl} ((A^{-1})^{l}_{m}\nabla D^m||_{L^{2,1}(B_1)} \le C||\nabla D||_{L^{2,1}(B_1)} \le C||\nabla u||_{L^2(B_1)}$ and hence, using Theorem $3.3.3$ of \cite{He1} (which implies that the standard $L^p$-theory extends to Lorentz spaces), we furthermore get
\[
 ||\nabla G^i_j||_{L^{2,1}(B_1)} \le C||\nabla u||_{L^2(B_1)}.
\]
By Theorem $3.3.4$ in \cite{He1} we conclude that
\begin{align*}
 ||G^i_j||_{L^{\infty}(B_1)}\le& C(||G^i_j||_{L^{2,1}(B_1)}+ ||\nabla G^i_j||_{L^{2,1}(B_1)})\\
\le& C||\nabla G^i_j||_{L^{2,1}(B_1)} \\
\le& C||\nabla u||_{L^2(B_1)},
\end{align*}
where we used again Theorem $3.3.3$ of \cite{He1} (which ensures that the Poincar$\acute{\text{e}}$'s inequality extends to Lorentz spaces) and the fact that $G^i_j=0$ on $\partial B_1$ in the second estimate.

Combining these estimates and using that $B^i_j=F^i_j+G^i_j$ gives
\[
 ||B^i_j||_{L^{\infty}(B_1)} \le C||\nabla u||_{L^2(B_1)}.
\]

Note that in the case $g^k_{jl}\equiv 0$ the proof simplifies since we don't need to use Lorentz spaces in order to get the desired result.
\end{proof}

\begin{corollary}\label{CORO1}
Under the assumptions of Theorem \ref{THMTL} there exist $a \in W^{1,2}(B_1, M_n(\mathbb{R}))$ and $b\in W^{1,2}(B_1, \mathbb{R}^n)$ such that
\begin{equation}
\Delta u\,=\, \nabla a \cdot \nabla^\perp b \quad \quad \text{in }\,\, B_1\,
\end{equation}
and
\[
 ||\nabla a||_{L^2(B_1)}+||\nabla b||_{L^2(B_1)} \le C||\nabla u||_{L^2(B_1)}.
\]
In particular $u$ is continuous in $B_1$.
\end{corollary}
\begin{proof}
Equation \eqref{AandB2} and the above estimates for the $L^\infty$-norms of $A$ and $B$ imply the existence of $\eta \in W^{1,2}(B_1, \mathbb{R}^n)$ such that
\[
 ||\nabla \eta||_{L^2(B_1)} \le C ||\nabla u||_{L^2(B_1)}
\]
and
\begin{align}
 \nabla^\perp \eta =A\nabla u +B \nabla^\perp u. \label{eta}
\end{align}
Multiplying this equation with $A^{-1}$ and taking the divergence yields
\begin{equation}
\Delta u^l\,=\, \nabla (A^{-1})^l_{k}\cdot \nabla^{\perp} \eta^k - \nabla (A^{-1} B)^l_{k}\cdot \nabla^{\perp} u^k\,,\quad l=1,2,...,n\,.
\end{equation}
The continuity of $u$ now follows from Wente's lemma \ref{Wen}.
\end{proof}

\subsection{Improved local estimate for the matrix $P$}\label{MatrixU}
We next show that $\Delta P$ also has a special Jacobian structure under the assumptions of Theorem \ref{THMTL}.
\begin{lemma}\label{P2}
Under the assumptions of Theorem \ref{THMTL} there exist $\xi\in W_0^{1,2}(B_1, so(n))$, $\eta\in W^{1,2}(B_1,\mathbb{R}^n)$ and $Q_k, R_k\in W^{1,2}(B_1, Gl_n(\mathbb{R})), k =1,...,n$ with
$$
\|\nabla \xi\|_{L^2(B_1)} +  \|\nabla \eta\|_{L^{2}(B_1)} \leq C\|\nabla u\|_{L^2(B_1)}$$
and
$$\sum_k \left(\|\nabla Q_k\|_{L^{2}(B_1)} + \|\nabla R_k\|_{L^{2}(B_1)}\right) \leq C
$$
such that
\begin{equation}\label{P4}
\Delta P \,= \,\nabla P \cdot\nabla^{\perp} \xi+ \nabla Q_k \cdot \nabla^{\perp} \eta^k + \nabla R_k \cdot \nabla^{\perp} u^k\,.
\end{equation}
\end{lemma}
\begin{proof}
Taking the divergence on both sides of equation \eqref{P-4} yields
\begin{equation} \label{P1}
\Delta P^i_{j} \,=\, \nabla P^i_m\cdot\nabla^{\perp} \xi^m_j - \text{div }(\Omega^i_z \,P^z_j)\,,\quad 1\leq i,j\leq n\,.
\end{equation}

Now combining $\Omega^i_z = f^i_{zl} \nabla u^l + g^i_{zl} \nabla^{\perp} u^l$ with equation \eqref{eta} gives
\begin{align*}
&\Delta P^i_{j} \,=\,\nabla P^i_m\cdot\nabla^{\perp} \xi^m_j - \text{div }(\Omega^i_z \,P^z_j)\\
= &\,\nabla P^i_m\cdot\nabla^{\perp} \xi^m_j - \text{div }\left[f^i_{zl}\,\left[(A^{-1})^{l}_{k}\nabla^{\perp} \eta^k - (A^{-1} B)^{l}_{k} \nabla^{\perp} u^k\right]\, P^z_j +  g^i_{zl} \nabla^{\perp} u^l  P^z_j\right]\\
=& \,\nabla P^i_m\cdot\nabla^{\perp} \xi^m_j - \nabla \left[f^i_{zl} P^z_j (A^{-1})^l_k\right]\cdot \nabla^{\perp} \eta^k + \nabla \left[f^i_{zl} P^z_j (A^{-1} B)^l_k\right]\cdot \nabla^{\perp} u^k \\
&- \nabla(g^i_{zl}P^z_j)\cdot \nabla^{\perp} u^l \,.\notag
\end{align*}
Defining $(Q_k)^i_j = -f^i_{zl} P^z_j (A^{-1})^l_k$ and $(R_k)^i_j=f^i_{zl} P^z_j (A^{-1} B)^l_k -g^i_{zl}P^z_j$ where $1\leq k, i, j\leq n$, completes the proof.
\end{proof}

Next, based on Lemma \ref{P2}, we prove a local estimate on the oscillation of the matrix $P$. As we shall see, the Jacobian structure of $\Delta P$ enters in a crucial way.
\begin{lemma}\label{L21EST}
Let $u$ and $\Omega$ satisfy the assumptions of Theorem \ref{THMTL}. Then for any $x\in B_1$, any $r>0$ such that $B_{2r}(x) \subset B_1$ and any $y\in B_{r}(x)$ we have
\begin{equation}\label{TEMO2011}
|P(y)- P(x)|\,\leq\, C \sqrt{\varepsilon_0}\,.
\end{equation}
\end{lemma}
\begin{proof}
Let $\widetilde{P} \in W^{1,2}(B_1, M_n(\mathbb{R}))$ be the weak solution of
\begin{equation}
  \left\{
   \aligned
   \Delta \widetilde{P}\,=\,& \,\nabla P \cdot\nabla^{\perp} \xi + \nabla Q_k \cdot \nabla^{\perp} \eta^k + \nabla R_k \cdot \nabla^{\perp} u^k  &&\text{in  } B_1 \,, \\
   \widetilde{P}  \, = \, &0  &&\text{on  }\partial B_1\,,\\
   \endaligned
  \right.
\end{equation}
where $Q_k$ and $R_k$ are as in Lemma \ref{P2}.

Then by Wente's lemma \ref{Wen} we have $\widetilde{P}\in C^0(B_1, M_n(\mathbb{R}))$ and
\begin{equation}\label{100}
\|\widetilde{P}\|_{L^\infty(B_1)}+\|\nabla \widetilde{P}\|_{L^2(B_1)} \,\leq \, C\sqrt{\varepsilon_0}\,.
\end{equation}

Since $\Delta (P - \widetilde{P}) = 0$ in $B_1$, we know that $V = P - \widetilde{P} \in C^{\infty}(B_1, M_n(\mathbb{R}))$ is harmonic. Now for any $x\in B_1$, any $r>0$ such that $B_{2r}(x) \subset B_1$ and any $y\in B_{r}(x)$ we have
\begin{align}\label{1000}
|V(y)-V(x)|\leq& Cr\|\nabla V\|_{L^{\infty}(B_r(x))} \notag\\
\leq& \, C  \|\nabla V\|_{L^2(B_{2r}(x))}\\
\,\leq\,& C \left(\|\nabla P\|_{L^2(B_{2r}(x))} + \|\nabla \widetilde{P}\|_{L^2(B_{2r}(x))}\right)\,\leq\,C \sqrt{\varepsilon_0} \,,\notag
\end{align}
where we used the mean value property of $V$ and \eqref{100}, \eqref{P-2}. Combining \eqref{100} and \eqref{1000} yields that for any $x\in B_1$, any $r>0$ such that $B_{2r}(x) \subset B_1$ and any $y\in B_{r}(x)$ we have
\begin{equation}
|P(y)- P(x)|\,\leq\, C \sqrt{\varepsilon_0}\,,
\end{equation}
which gives the desired result.
\end{proof}

\section{Proof of Theorem \ref{THMTL}}\label{LastSec}
With the results of section \ref{4} at our disposal, we are now in a position to prove Theorem \ref{THMTL}. The local estimate on the oscillation of the transformation matrix $P$ in Lemma \ref{L21EST} turns out to be the key ingredient of our proof.

\begin{proof}(of Theorem \ref{THMTL})
Using Theorem \ref{PCloseTOID}, Theorem \ref{PCloseTOID299} and Proposition \ref{XI}, for any $x\in B_1$, any $r>0$ such that $B_{2r}(x) \subset B_1$ and any $y\in B_{r}(x)$ we have (choosing $\varepsilon_0$ sufficiently small)
\begin{align}
0&\leq \,\frac{1}{2}|\nabla u|^2(y) \leq \left(A\nabla u + B\nabla^{\perp} u\right) \cdot (P^T\nabla u)(y)\notag\\
&=  \left(A\nabla u + B\nabla^{\perp} u\right)\cdot \left[\left(P^T(x) + \left(P^T-P^T(x)\right)\right)\nabla u\right](y)\,,
\end{align}
and therefore by Lemma \ref{L21EST} and \eqref{eta}
\begin{align}
&\nabla^\perp\eta \cdot \left(P^T(x)\nabla u\right)(y) = \left(A\nabla u + B\nabla^{\perp} u\right) \cdot \left(P^T(x)\nabla u\right)(y)\notag\\
\geq &\,\frac{1}{2}|\nabla u|^2(y) - \left(A\nabla u + B\nabla^{\perp} u\right)\cdot \left[\left(P^T-P^T(x)\right)\nabla u\right](y)\geq \,\frac{1}{4}|\nabla u|^2(y)\,.\label{PPP}
\end{align}
Now we choose a function
\begin{equation}\label{funcphi}\phi \in C^{\infty}_0(B_1) \text{ with } \phi\geq 0, \,\text{spt}(\phi) \subseteq B_{\frac{1}{2}}, \,\phi =2 \text{ on } B_{\frac{3}{8}}, \text{ and } \int_{B_1} \phi \,dx =1\,.
\end{equation}
Moreover, we additionally assume that $\|\nabla \phi\|_{L^\infty(B_1)}\leq 100$. Using \eqref{PPP}, one verifies directly that (using Definition \ref{10003})
\begin{align*}
\||\nabla u|^2\|_{h^1(B_1)}=& \,\int_{B_1} \sup_{0<t< 1-|x|} \phi_t \ast |\nabla u|^2dx\\
\leq& \,4\int_{B_1} \sup_{0<t< 1-|x|}\phi_t \ast \left(\nabla^\perp\eta \cdot (P^T(x)\nabla u)\right)dx\\
= & \,4\int_{B_1} \sup_{0<t< 1-|x|}\phi_t \ast \left[(P^T(x))_{ij}\left(\nabla^{\perp}\eta^i \cdot \nabla u^j\right)\right]dx\\
\leq &\,C \,\sum_{i,j=1}^n\|\nabla^\perp \eta^i \cdot \nabla u^j\|_{h^1(B_1)}\\
\leq &\, C \|\nabla^{\perp}\eta\|_{L^2(B_1)}\|\nabla u\|_{L^2(B_1)} \leq C \int_{B_1}|\nabla u|^2,
\end{align*}
where we have used the fact
$$
\nabla^\perp \eta^i \cdot \nabla u^j \in h^1(B_1) \quad \text{and} \quad \|\nabla^\perp \eta^i \cdot \nabla u^j\|_{h^1(B_1)}\leq C\|\nabla \eta\|_{L^2(B_1)}\|\nabla u\|_{L^2(B_1)}
$$
for all $i, j = 1,2,...,n.$ To see this, we first extend $\eta^i - \frac{1}{|B_1|}\int_{B_1}\eta^i$ and $u^j - \frac{1}{|B_1|}\int_{B_1}u^j$ from $B_1$ to $\mathbb{R}^2$ which yields the existence of $\tilde{\eta}^i, \tilde{u}^j \in W^{1,2}_0 (\mathbb{R}^2)$ such that
\begin{equation}
\int_{\mathbb{R}^2}|\nabla \tilde{\eta}^i|^2 \leq C \int_{B_1}|\nabla \eta^i|^2\quad \text{and}\quad \int_{\mathbb{R}^2}|\nabla \tilde{u}^j|^2 \leq C \int_{B_1}|\nabla u^j|^2
\end{equation}
and
\begin{equation}\label{10006}
\nabla \tilde{\eta}^i = \nabla \eta^i \quad \text{and} \quad \nabla \tilde{u}^j = \nabla u^j \quad \text{a.e. in } B_1\,.
\end{equation}
Then by the results of \cite{CLMS} we know that
\begin{align}
\|\nabla^\perp \tilde{\eta}^i \cdot \nabla \tilde{u}^j\|_{\mathcal{H}^1(\mathbb{R}^2)}:&= \int_{\mathbb{R}^2} \sup_{\phi\in \mathcal{T}}\sup_{t>0}\left| \int_{B_t(x)}\frac{1}{t^2}\phi\left(\frac{x-y}{t}\right)\left(\nabla^\perp \tilde{\eta}^i \cdot \nabla \tilde{u}^j\right)(y)dy\right|dx\notag\\
& \leq C \|\nabla \tilde{\eta}^i\|_{L^2(\mathbb{R}^2)}\|\nabla \tilde{u}^j\|_{L^2(\mathbb{R}^2)} \leq C \|\nabla \eta\|_{L^2(B_1)}\|\nabla u\|_{L^2(B_1)}\,,\label{10007}
\end{align}
where $\mathcal{T} = \{\phi \in C^{\infty}(\mathbb{R}^2): \text{spt}(\phi)\subset B_1 \text{ and } \|\nabla \phi\|_{L^\infty} \leq 100 \}$. By \eqref{10006}, \eqref{10007} and Definition \ref{10003}, it is clear that
\begin{align}\label{TEMP909}
\|\nabla^\perp \eta^i \cdot \nabla u^j\|_{h^1(B_1)} &= \|\nabla^\perp \tilde{\eta}^i \cdot \nabla \tilde{u}^j\|_{h^1(B_1)}\notag\\
&\leq \|\nabla^\perp \tilde{\eta}^i \cdot \nabla \tilde{u}^j\|_{\mathcal{H}^1(\mathbb{R}^2)} \leq C \|\nabla \eta\|_{L^2(B_1)}\|\nabla u\|_{L^2(B_1)}.
\end{align}
This completes the proof of the theorem.
\end{proof}
In view of the discussions in Section \ref{Sec1}, this also yields an alternative proof of Theorem \ref{THM1}.

\appendix

\section{Wente's lemma and the local Hardy space}\label{APPDEN}

\subsection{Wente's lemma}\label{WENTYLEMMA}

An important ingredient in our estimates is Wente's lemma, see e.g. \cite{BC} and \cite{He1}, and for a generalized version, see e.g. \cite{BG} and \cite{CL}.
\begin{lemma} (\cite{W})\label{Wen}
Let $a,b\in W^{1,2}(B_1,\mathbb{R})$ and let $w$ be the solution of
\begin{equation}
\left\{
\begin{aligned}
\Delta w\, &=\,\frac{\partial a}{\partial y}\frac{\partial b}{\partial x} - \frac{\partial a}{\partial x}\frac{\partial b}{\partial y} \,=\, \nabla a \cdot \nabla^{\perp} b && \text{in } \, B_1\,,\\
w\, &=\,0 \quad \text{or} \quad \frac{\partial w}{\partial \nu}\,=\,0 && \text{on } \, \partial B_1\,.
\end{aligned}
\right.
\end{equation}
Then $w\in C^0\cap W^{1,2}(B_1,\mathbb{R})$ and the following estimate holds
\begin{equation}
\|w\|_{L^\infty(B_1)}+ \|\nabla w\|_{L^2(B_1)}\,\leq\,C\|\nabla a\|_{L^2(B_1)}\|\nabla b\|_{L^2(B_1)}\,,
\end{equation}
where we choose $\int_{B_1} w = 0 $ for the Neumann boundary data. Here $\nabla = (\partial_{x}, \partial_{y})$ and $\nabla^{\perp} = (-\partial_{y}, \partial_{x})$.
\end{lemma}

\begin{remark}\label{Wen2}
In fact, by Theorem $3.4.1$ in \cite{He1}, we also have the following estimate for $\nabla w$ in $L^{2,1}$:
\begin{equation}
\|\nabla w\|_{L^{2,1}(B_1)}\,\leq\, C \|\nabla a\|_{L^2(B_1)}\|\nabla b\|_{L^2(B_1)}\,.
\end{equation}
\end{remark}

\subsection{Local Hardy space $h^1(B_1)$}\label{HardySpace1}
Let us first recall the definition of the local Hardy space $h^1(B_1)$.
\begin{definition}\label{10003}(\cite{Mi}) Choose a Schwartz function $\phi \in C^{\infty}_0(B_1)$ such that $\int_{B_1} \phi \,dx =1$ and let $\phi_t(x) = t^{-2}\phi\left(\frac{x}{t}\right)$. For a measurable function $f$ defined in $B_1$ we say that $f$ lies in the local Hardy space $h^1(B_1)$ if the radial maximal function of $f$
\begin{equation}\label{Maxima}
f^{\ast}(x) = \sup_{0<t< 1-|x|}\left| \int_{B_t(x)}\frac{1}{t^2}\phi\left(\frac{x-y}{t}\right) f(y)dy\right|= \sup_{0<t< 1-|x|}\left| \phi_t \ast f\right|(x)
\end{equation}
belongs to $L^1(B_1)$ and we define
\begin{equation}
\|f\|_{h^1(B_1)} = \|f^{\ast}(x)\|_{L^1(B_1)}\,.
\end{equation}
It follows immediately that $h^1(B_1)$ is a strict subspace of $L^1(B_1)$ and $\|f\|_{L^1(B_1)}\leq \|f\|_{h^1(B_1)}$.
\end{definition}

Next we state an existence and regularity result for boundary value problems in the local Hardy space, which follows along the lines of a corresponding result in \cite{Sem}. For a more general version we refer to Chang, Krantz and Stein's work in \cite{CKS}{\footnote{By the results of \cite{Mi}, the local Hardy space $h^1(B_1)$ is just the local Hardy space $\mathcal{H}^1_r(B_1) \subset h^1_r(B_1)$ (a function in $\mathcal{H}^1_r(B_1)$ is obtained by restricting a function in $\mathcal{H}^1(\mathbb{R}^2)$ to $B_1$) in \cite{CKS}, therefore \cite[Theorem 5.1]{CKS} is applicable in our case.}}. In order to make the paper self-contained, we decided to include the proof of the result below.

\begin{theorem}\label{CDSTHM}(cf. \cite[Theorem 1.100]{Sem} and \cite[Theorem 5.1]{CKS})
Let $f\in h^1(B_1)$ and assume that $f\geq 0$ a.e. in $B_1$. Then there exists a function $\psi \in L^\infty\cap W^{1,2}(B_1)$ solving the Dirichlet problem
 \begin{equation}
  \left\{
   \begin{aligned}
     \Delta \psi\,&=\, f && \text{in }\, B_1\,, \\
    \psi\,&=\, 0 && \text{on }\, \partial B_1\,.\\
   \end{aligned}
 \right.
   \end{equation}
Moreover, there exists a constant $C>0$ such that
\begin{equation}\label{psii1}
\|\psi\|_{L^{\infty}(B_1)}  + \|\nabla \psi\|_{L^2(B_1)} \,\leq\, C \,\|f\|_{h^1(B_1)}\,.
\end{equation}
\end{theorem}
\begin{proof}
The idea of the proof follows \cite[Proposition 1.68]{Sem}. Since the Green's
function of $\Delta$ on $B_1$ is given by $\frac{1}{2\pi}\ln|x|$, for $x\in B_1$, we can write
\begin{align}\label{psicon}
\psi(x) &= \frac{1}{2\pi}\int_{B_1} f(y) \left(\ln|x-y| - \ln\left(\left|\frac{x}{|x|} - |x|y\right|\right)\right)dy\,.
\end{align}

Let $\theta\in C^{\infty}_0(B_1)$ be a smooth bump function such that $0\leq \theta \leq 1$, $\theta=1$ in $B_{\frac{1}{16}}$ and $\text{spt}(\theta)\subset B_{\frac{1}{8}}$. For $x\in B_1$ we define
\begin{equation}\label{lfunc}
l_x(y) = \sum_{j=0}^{\infty}\theta\left(2^j(1-|x|^{-1})(x-y)\right) \quad \text{for }y\in B_1\,.
\end{equation}
We claim that for any $x,y\in B_1$
\begin{equation}\label{Lest2}
-20\ln2 \leq \ln|x-y| - \ln\left(\left|\frac{x}{|x|} - |x|y\right|\right) + l_x(y)\ln2  \leq 20\ln 2\,,
 \end{equation}
To see this, it is clear that for $x,y\in B_1$ such that
\begin{equation}\label{Temp-9090-2}
2^{-k}\leq |x-y|\leq 2^{-k+1}, \quad k=\N_0
\end{equation}
we have
\begin{equation}\label{theta4409909}
-k\ln 2\,\leq \ln|x-y|\,\leq (-k+1)\ln 2\,.
\end{equation}
Now note that
$$
1-|x|-|x-y| \leq 1-|x|+|x|-|y|= 1-|y| \leq 1-|x|+|x-y|\,,
$$
and therefore for $x\in B_{1-2^{-i-1}} \setminus  B_{1-2^{-i}}$, i.e., $1-|x| \in [2^{-i-1},2^{-i}], i=\N_0$ (with $\bar{B}_0 = \emptyset$) and any $y\in B_1$ satisfying \eqref{Temp-9090-2}, we have
$$
1-|y| \,\in\,
\left\{
   \begin{aligned}
     &\left[2^{-i-1}-2^{-k+1},2^{-i}+2^{-k+1}\right]\quad &&\text{if } k\geq i+4;\\
    &\left[0,2^{-i}+2^{-k+1}\right]\quad &&\text{if } k\leq i+3.\\
   \end{aligned}
 \right.
$$
We also have
\begin{align}\label{TEMO2000009}
0 \leq (1-|x|)(1-|y|) &\leq (1-|x|^2)(1-|y|^2) \notag\\
   &= \left|\frac{x}{|x|}- |x|y\right|^2 - |x-y|^2 \leq 2^2(1-|x|)(1-|y|)\,,
\end{align}
and thus
$$
\left|\frac{x}{|x|}- |x|y\right|^2 - |x-y|^2 \in
\left\{
   \begin{aligned}
     &\left[2^{-2i-2}-2^{-i-k},2^{-2i+2}+2^{-i-k+3}\right]\quad &&\text{if } k\geq i+4;\\
    &\left[0,2^{-2i+2}+2^{-i-k+3}\right]\quad &&\text{if } k\leq i+3.\\
   \end{aligned}
 \right.
$$
Combining this with \eqref{Temp-9090-2} we get
$$
\left|\frac{x}{|x|}- |x|y\right|^2 \in\left\{
   \begin{aligned}
     &\left[2^{-2i-2}-2^{-i-k} + 2^{-2k},2^{-2i+2}+2^{-i-k+3}+2^{-2k+2} \right] \text{if } k\geq i+4;\\
    &\left[2^{-2k},2^{-2i+2}+2^{-i-k+3}+2^{-2k+2}\right]\,\ \text{if } k\leq i+3.\\
   \end{aligned}
 \right.
$$

Now using the facts that for $k\geq i+4$ we have
$$
2^{-2i-2}-2^{-i-k} + 2^{-2k} \geq 2^{-2i-4} \quad \text{and}\quad 2^{-2i+2}+2^{-i-k+3}+2^{-2k+2}  \leq 2^{-2i+4}
$$
and for $k \leq i+3$ we have
$$
2^{-2i+2}+2^{-i-k+3}+2^{-2k+2}  \leq 2^{-2k+10}\,,
$$
we arrive at
$$
\left|\frac{x}{|x|}- |x|y\right|^2 \,\in\, \left\{
   \begin{aligned}
     &\left[2^{-2i-4},2^{-2i+4}\right]\quad &&\text{if } k\geq i+4;\\
    &\left[2^{-2k},2^{-2k+10}\right]\quad &&\text{if } k\leq i+3,\\
   \end{aligned}
 \right.
$$
and hence
\begin{equation}\label{Lfunction1}
-\ln \left|\frac{x}{|x|}- |x|y\right|\, \in \,\left\{
   \begin{aligned}
     &[(i-2)\ln2, (i+2)\ln2]\quad &&\text{if } k\geq i+4;\\
    &[(k-5)\ln2, k\ln2]\quad &&\text{if } k\leq i+3.\\
   \end{aligned}
 \right.
\end{equation}
Combining \eqref{theta4409909} and \eqref{Lfunction1} we get
\begin{equation}\label{Lfunction1-1-1}
\ln|x-y| - \ln\left(\left|\frac{x}{|x|} - |x|y\right|\right) \in \left\{
   \begin{aligned}
     &[(-k+i-2)\ln2, (-k+i+3)\ln2]\quad &&\text{if } k\geq i+4;\\
    &[-5\ln2, \ln 2] \quad(\text{in fact, }\,[-5\ln2, 0])\quad &&\text{if } k\leq i+3,\\
   \end{aligned}
 \right.
\end{equation}
for any $x\in B_{1-2^{-i-1}} \setminus  B_{1-2^{-i}}, i\geq 0$, and any $y\in B_1$ satisfying \eqref{Temp-9090-2} for some $k\geq 0$.

Now for any $x\in B_{1-2^{-i-1}} \setminus  B_{1-2^{-i}}, i\geq 0$, and any $y\in B_1$ satisfying \eqref{Temp-9090-2}, since $0\leq \theta \leq 1$, $\theta=1$ in $B_{\frac{1}{16}}$ and $\text{spt}(\theta)\subset B_{\frac{1}{8}}$, we get that for any $j\geq 0$
$$
\theta\left(2^j(1-|x|)^{-1}(x-y)\right)= 0  \quad \text{for}\quad  |x-y|\ge  2^{-j-3}(1-|x|) \in [2^{-j-i-4}, 2^{-j-i-3}]
$$
and
$$
\theta\left(2^j(1-|x|)^{-1}(x-y)\right) = 1 \quad \text{for}\quad  |x-y|\leq 2^{-j-4}(1-|x|)\in [2^{-j-i-5}, 2^{-j-i-4}]\,.
$$
Therefore (combining with \eqref{Temp-9090-2}),
\begin{equation}\label{theta229090}
\theta\left(2^j(1-|x|)^{-1}(x-y)\right)= 0  \quad \text{for}\quad j \geq k-i-3
\end{equation}
and
\begin{equation}\label{theta110909}
\theta\left(2^j(1-|x|)^{-1}(x-y)\right)= 1  \quad \text{if}\quad k-1 \geq j+ i+5 \quad(\text{i.e. }\, j\leq k-i-6)\,.
\end{equation}

Hence for any $x\in B_{1-2^{-i-1}} \setminus  B_{1-2^{-i}}, i\geq 0$ and any $y\in B_1$ such that $2^{-k}\leq |x-y|\leq 2^{-k+1}$ for some $k=0,1,2,...,$ \eqref{lfunc}, \eqref{theta229090} and \eqref{theta110909} imply
\begin{equation}\label{theta559999999}
\left\{
   \begin{aligned}
     &k-i - 10\,\leq\,l_x(y)\,\leq\, k-i+10 \quad &&\text{if } k\geq i+4;\\
    & l_x(y)=0 \quad &&\text{if } k\leq i+3\,.\\
   \end{aligned}
 \right.
\end{equation}
Combining \eqref{Lfunction1-1-1} and \eqref{theta559999999} gives \eqref{Lest2}.

Therefore, in order to obtain the $L^\infty$-bound of $\psi$ on $B_1$ as in \eqref{psii1}, it suffices to bound $\int_{B_1} f(y) l_x(y) dy$ since we have \eqref{psicon}, \eqref{Lest2} and $\|f\|_{L^1(B_1)}\leq \|f\|_{h^1(B_1)}$.

In order to bound $\int_{B_1} f(y) l_x(y) dy$, we next claim that for any $x\in B_1$, $j\geq 0$ and $z\in B_{2^{-j-4}(1-|x|)}(x)$ we have
\begin{equation}\label{TempLL1}
\int_{B_1} f(y) 2^{2j+2}(1-|x|)^{-2}\theta\left(2^j(1-|x|)^{-1}(x-y)\right)dy \leq \int_{B_t(z)}\frac{1}{t^2}\phi\left(\frac{z-y}{t}\right)f(y)dy\,,
\end{equation}
where
$$t = 2^{-j-1}(1-|x|)$$
and $\phi$ is a nonnegative Schwartz function as in \eqref{funcphi}. To see \eqref{TempLL1}, we first note that since $\text{spt}(\theta)\subset B_{\frac{1}{8}}$, we have for any $x\in B_1$ and $j\geq 0$
\begin{align}\label{TempLL2}
&\int_{B_1} f(y) 2^{2j+2}(1-|x|)^{-2}\theta\left(2^j(1-|x|)^{-1}(x-y)\right)dy \notag\\
= &\int_{B_{2^{-j-3}(1-|x|)}(x)} f(y)2^{2j+2}(1-|x|)^{-2}\theta\left(2^j(1-|x|)^{-1}(x-y)\right)dy\,.
\end{align}
Now since $\frac{3}{8}(2^{-j-1}) = 2^{-j-4} + 2^{-j-3}$ and $2^{-j-4} + 2^{-j-1} = \frac{9}{16}2^{-j}<1$ for any $j\geq 0$, we see that for any $z\in B_{2^{-j-4}(1-|x|)}(x)$
\begin{equation}\label{TEMP2014}
B_{2^{-j-3}(1-|x|)}(x)\subseteq B_{\frac{3t}{8}}(z) \subset B_t(z)= B_{2^{-j-1}(1-|x|)}(z)\subseteq B_1\,.
\end{equation}
Using the facts that $f\geq 0, \,0\leq \theta \leq 1, \,\phi\geq0$ and $\phi = 2$ on $B_{\frac{3}{8}}$ we conclude
\begin{align*}
&\int_{B_{2^{-j-3}(1-|x|)}(x)} f(y) 2^{2j+2}(1-|x|)^{-2}\theta(2^j(1-|x|)^{-1}(x-y))dy\\
\leq &\int_{B_{2^{-j-3}(1-|x|)}(x)} f(y) 2^{2j+2}(1-|x|)^{-2}dy\leq \int_{B_{\frac{3t}{8}}(z)} f(y) 2^{2j+2}(1-|x|)^{-2}dy \\
\leq &\int_{ B_t(z)} f(y) 2^{2j+2}(1-|x|)^{-2} \phi\left(\frac{z-y}{t}\right)dy = \int_{B_t(z)}\frac{1}{t^2}\phi\left(\frac{z-y}{t}\right)f(y)dy\,.
\end{align*}
Combining this with \eqref{TempLL2} gives \eqref{TempLL1}. Therefore by \eqref{TempLL1} and the definition \eqref{Maxima} of the radial maximal function $f^{\ast}$, for any $x\in B_1$ and $j\geq 0$ we have
$$
\left| \int_{B_1} f(y) \theta(2^j(1-|x|)^{-1}(x-y))dy\right| \leq 2^{-2j-2}(1-|x|)^{2} \inf_{z\in B_{2^{-j-4}(1-|x|)}(x)} f^{\ast}(z)\,.
$$

Therefore, by \eqref{lfunc}, for any $x\in B_1$ we have
\begin{align}\label{FINALLLY}
&\left| \int_{B_1} f(y) l_x(y)dy\right| \leq \sum_{j=0}^{\infty}\left| \int_{B_1} f(y) \theta(2^j(1-|x|)^{-1}(x-y))dy\right|\notag\\
\leq \,& \sum_{j=0}^{\infty} 2^{-2j-2} (1-|x|)^{2}\inf_{z\in B_{2^{-j-4}(1-|x|)}(x)} f^{\ast}(z)\notag\\
 \leq \,& \frac{2^8}{3\pi}\sum_{j=0}^{\infty} \int_{ B_{2^{-j-4}(1-|x|)}(x) \setminus  B_{2^{-j-5}(1-|x|)}(x)} f^{\ast} (z)dz\\
\leq \,& \frac{2^8}{3\pi}\int_{B_1} f^{\ast}(z)dz \,\leq \,\frac{2^8}{3\pi}\|f\|_{h^1(B_1)}\,.\notag
\end{align}
Combining \eqref{TEMO2000009}, \eqref{Lest2} and \eqref{FINALLLY} yields (using $\|f\|_{L^1(B_1)}\leq \|f\|_{h^1(B_1)}$)
$$
|\psi(x)| = -\frac{1}{2\pi}\int_{B_1} f(y) \left(\ln|x-y| - \ln\left(\left|\frac{x}{|x|} - |x|y\right|\right)\right)dy \leq  C\|f\|_{h^1(B_1)}\,.
$$
This gives the desired $L^{\infty}$-bound of $\psi$ on $B_1$. The $L^2$-estimate for $\nabla \psi$ simply follows from an integration by parts argument as in the proof of Wentes's lemma.
\end{proof}


\begin{thebibliography}{10}

\bibitem[BC84]{BC} H. Brezis and J. Coron, Multiple solutions of $H$-systems and Rellich's conjecture, Commun. Pure Appl. Math. 37(2) (1984), 149--187.

\bibitem[BG93]{BG} F. Bethuel and J. Ghidaglia, Improved regularity of solutions to elliptic equations involving Jacobians and applications, J. Math. Pures Appl. (9) 72 (1993), no. 5, 441--474.

\bibitem[CL92]{CL} S. Chanillo and Y. Li, Continuity of solutions of uniformly elliptic equations in $\mathbb{R}^2$, Manuscripta Mathematica, 77 (1992), 415--434.

\bibitem[Ch95]{Cho} P. Chone, A regularity result for critical points of conformally invariant functionals, Potential Anal., 4 (1995), 269--296.

\bibitem[CKS93]{CKS}  D. Chang, S. G. Krantz and E.M. Stein, $H^p$ theory on a smooth domain in $\mathbb{R}^N$ and elliptic boundary value problems, J. Funct. Anal. 114 (1993), no. 2, 286--347.

\bibitem[CLMS93]{CLMS} R. R. Coifman, P.-L. Lions,  Y. Meyer and S. Semmes, Compensated compactness and Hardy spaces, J. Math. Pures Appl. (9) 72 (1993), no. 3, 247--286.

\bibitem[CM08-1]{CM1} T.H. Colding and W.P. Minicozzi II, Width and finite extinction time of Ricci flow, Geometry and Topology, 12 (2008), no. 5, 2537--2586.

\bibitem[CM08-2]{CM2} T.H. Colding and W.P. Minicozzi II, Width and mean curvature flow, Geometry and Topology, 12 (2008), no. 5, 2517--2535.




\bibitem[He02]{He1} F. H\'elein,   \textit{Harmonic maps, conservation laws, and moving frames,
Cambridge Tracts in Mathematics}, 150., Cambridge University Press, Cambridge, 2002.


\bibitem[IM01]{IM} T. Iwaniec and G. Martin, \textit{Geometric Function Theory and Non-linear Analysis}, Oxford University Press, Clarendon, 2001.


\bibitem [LR08]{LamR} T. Lamm and T. Rivi\`ere, Conservation laws for fourth order systems in four dimensions, Communications in Partial Differential Equations, v.33 (2008), no.2, 245--262.

\bibitem[Lin12]{Lin12} L. Lin, Uniformity of harmonic map heat flow at infinite time, preprint 2012.

\bibitem[Mi90]{Mi} A. Miyachi, $H^p$ spaces over open subsets of $\mathbb{R}^n$, Studia Math. 95 (1990), no. 3, 205--228.



\bibitem[Qi93]{Q} J. Qing, Boundary regularity of weakly harmonic maps from surfaces, J. Funct. Anal., 114 (1993), 458--466.

\bibitem[Ri07]{Riv1} T. Rivi\`ere, Conservation laws for conformally invariant variational problems, Invent. Math., 168 (2007), 1--22.

\bibitem[Ri08]{Riv2} T. Rivi\`ere, Integrability by Compensation in the Analysis of Conformally Invariant Problems, Preprint.

\bibitem[Sc10]{Sc} A. Schikorra, A remark on Gauge transformations and the moving frame method, Annales de l'Institut Henri Poincar$\acute{\text{e}}$ - Analyse non lin$\acute{\text{e}}$aire, 27 (2010), 503--515.

\bibitem[Se94]{Sem} S. Semmes. A primer on Hardy spaces, and some remarks on a theorem of Evans and M$\ddot{\text{u}}$ller, Commun. Partial Differ. Equations, 19(1--2), 277--319, 1994.

\bibitem[ST11]{ST} B. Sharp and P. Topping, Decay estimates for Rivi\`ere's equation, with applications to regularity and compactness, to appear in Trans. Amer. Math. Soc.




\bibitem[Ta85]{Ta2} L. Tartar, Remarks on oscillations and Stokes' equation. In: Frisch, U., et al. (eds.) Macroscopic Modelling of Turbulent Flows, vol. 230, pp. 24--31. Nice, 1984. Lecture Notes Phys., Springer, Berlin (1985).



\bibitem[Uh82]{Uh} K. Uhlenbeck, Connections with $L^{p}$ bounds on curvature, Comm. Math. Phys. v. 83 (1982), no. 1, 31--42.

\bibitem[We69]{W} H. Wente, An existence theorem for surfaces of constant mean curvature, J. Math. Anal. Appl., 26 (1969), 318-344.

\end{thebibliography}
\end{document}